\documentclass[11pt]{article}

 \usepackage{array}
 \usepackage{enumitem}
\newcolumntype{C}[1]{>{\centering\arraybackslash}m{#1}}
\usepackage{subfig}
\usepackage{graphicx}
\usepackage{amssymb}
\usepackage{mathtools}
\usepackage{amsmath}
\usepackage{amsthm}
\usepackage{latexsym}
\usepackage{amsfonts}
\usepackage{etoolbox}
\usepackage{multirow}
\usepackage{float}
\usepackage{amsmath,color}
\usepackage{url}
\usepackage{hyperref}
\usepackage[T1]{fontenc}
\usepackage{doi}

\usepackage{arydshln}
\usepackage{lipsum}

\newtheorem{theorem}{Theorem}{}
\newtheorem{lemma}{Lemma}{}
{}
{}
{}
{}
{}

\newcommand{\rk}{\operatorname{rank}}

\theoremstyle{definition}
\usepackage[utf8]{inputenc}
\usepackage{authblk}

\usepackage{tikz}
\usetikzlibrary{positioning}

\providecommand{\keywords}[1]{%
  \vspace{2mm}
  {\small{\em Keywords:} #1}
}

\title{\bf Multiplicity Bounds for Arbitrary Eigenvalues of Connected Signed Graphs}

\author[1,2]{Monther R. Alfuraidan}
\author[2,*]{Suliman Khan}
\affil[1]{{\small Department of Mathematics, King Fahd University of Petroleum \& Minerals, Dhahran 31261, Saudi Arabia}}

\affil[2]{{\small Interdisciplinary Center of Smart Mobility and Logistics, King Fahd University of Petroleum \& Minerals, Dhahran 31261, Saudi Arabia\\
\texttt{monther@kfupm.edu.sa, suliman5344@gmail.com}}}

\date{}
\begin{document}

\maketitle

\vspace{-0.1cm}

\begin{abstract}
The study of eigenvalue multiplicities plays a central role in the spectral theory of signed graphs, extending several classical results from the unsigned setting. While most existing work focuses on the nullity of a signed graph (the multiplicity of the eigenvalue $0$), much less is known for arbitrary eigenvalues. In this paper we establish a sharp upper bound for the multiplicity $m(G_\sigma,\lambda)$ of any real eigenvalue $\lambda$ of a connected signed graph $G_\sigma$ in terms of its girth.

Our main result shows that
\[
m(G_\sigma,\lambda) \le n - g(G_\sigma) + 2,
\]
where $n$ is the number of vertices and $g(G_\sigma)$ is the girth. We prove that equality holds if and only if $G_\sigma$ is switching equivalent to one of the following extremal families:
(i) a balanced complete graph and $\lambda=-1$;
(ii) an antibalanced complete graph and $\lambda=1$; or
(iii) a balanced complete bipartite graph and $\lambda=0$.
This completely extends and generalizes the known result for the nullity case ($\lambda=0$), originally due to Wu et al.~(2022), to the full eigenvalue spectrum.

Our approach combines Cauchy interlacing, switching equivalence, and a structural analysis of induced cycles in signed graphs. We also provide a detailed characterization of eigenvalues with multiplicity $1$ and $2$ for signed cycles. Several examples are included to illustrate the sharpness and the spectral behavior of the extremal families.
\end{abstract}

\section{Introduction}
\indent

Throughout this paper, we assume that all graphs are simple and undirected, containing neither loops nor multiple edges. Let $G=(V,E)$ be a simple graph of order $|V|=n$. We write $v_i \sim v_j$ to indicate that the vertices $v_i$ and $v_j$ in $G$ are adjacent. The neighborhood of a vertex $v_i$ in $G$ is defined as $N_G(v_i)=\{\,v_j\in V(G): v_j\sim v_i \}$. A graph $H$ is called a \emph{subgraph} of $G$ if $V(H)\subseteq V(G)$ and $E(H)\subseteq E(G)$. Furthermore, $H$ is an \emph{induced subgraph} of $G$ if for each $v \in V(H)$, we have $N_H(v) = N_G(v) \cap V(H)$. Equivalently, $v_i$ and $v_j$ are adjacent in $H$ if and only if they are adjacent in $G$. The \emph{adjacency matrix} $A(G)=[a_{ij}]$ of $G$ is defined by $a_{ij}=1$ whenever $v_i \sim v_j$ and $0$ otherwise. Since $A(G)$ is real symmetric, all eigenvalues are real and can be ordered as $\lambda_1 \ge \lambda_2 \ge \dots \ge \lambda_n$. The \emph{spectrum} of $G$, denoted $\mathrm{Spec}_A(G)$, is the multiset of eigenvalues of $A(G)$. The \emph{rank} of $G$, denoted $\operatorname{rk}(G)$, is the rank of $A(G)$. For an eigenvalue $\lambda$ of $A(G)$, its \emph{multiplicity} is denoted by $m(G,\lambda)$.

A \emph{signed graph} is a natural extension of an ordinary graph in which each edge carries a positive or negative sign. Formally, a signed graph $G^\sigma = (G,\sigma)$ consists of an underlying graph $G=(V,E)$ together with a sign function $\sigma : E \to \{+1,-1\}$, where edges with $\sigma(e)=+1$ are positive and those with $\sigma(e)=-1$ are negative. When all edges are positive (resp.\ negative), we write $(G,+)$ (resp.\ $(G,-)$). Signed graphs capture systems where interactions may be cooperative or antagonistic, and they naturally incorporate notions such as balance, switching equivalence, and parity of cycles.

\medskip
\noindent{\bf Background and classical results.}
The multiplicities of eigenvalues of Hermitian and real symmetric matrices associated with graphs have been studied extensively. Parter--Wiener theory, developed in \cite{Parter1960,Wiener1984} and later extended in \cite{Johnson2003,Johnson2006}, provides important tools for understanding how graph structure influences eigenvalue multiplicity. Star complement techniques have been fruitfully applied to obtain upper bounds on eigenvalue multiplicities \cite{Bell2003}. Multiplicities of adjacency eigenvalues also carry structural meaning. For example, the nullity (multiplicity of $0$) plays a central role in chemical graph theory \cite{Atkins2006}. Among connected graphs of order $n\ge 2$, the eigenvalue $0$ attains its maximum multiplicity $n-2$ in complete bipartite graphs $K_{n_1,n_2}$, while the eigenvalue $-1$ attains multiplicity $n-1$ in the complete graph $K_n$. Motivated by these extremal cases, several works---including \cite{Petrovic1991,Yang2019,Johnson2009,Rowlinson2013,Rowlinson2014,Bang2017}---have investigated graphs with large multiplicities of specific eigenvalues.

\medskip
\noindent{\bf Structural parameters and eigenvalue multiplicities.}
A parallel line of research relates eigenvalue multiplicities to structural parameters such as girth, matching number, cyclomatic number, diameter, and pendant vertices. Cámara and Haemers \cite{Camara2014} showed that multiplicity patterns characterize almost complete graphs. Cioabă et al.\ \cite{Cioab2015} investigated sparse graphs with all but two eigenvalues equal to $\pm 1$. Wang \cite{Wang2022m} obtained multiplicity bounds in terms of diameter. Ma et al.\ \cite{Ma2016} and Tam and Huang \cite{Tam2017} linked nullity to pendant vertices and the cycle space, while Wang et al.\ \cite{Wang2020a} extended these ideas to arbitrary eigenvalues. In \cite{Wang2020}, the multiplicity of $-1$ was related to pendant vertices in trees. Further results related to girth, extremal multiplicity, and eigenvalue structure appear in \cite{Zhou2021,Rowlinson2011,Chang2020,Cvetkovic1995,Horn1985,Wang2022,Chang2022a,Lu2021a,Fan2013,Fan2014,Wong2022,Chang2022}.

\medskip
\noindent{\bf Signed graphs and spectral multiplicities.}
Since signed graphs generalize ordinary graphs, it is natural to examine how switching equivalence, balance, and signed cycle structure affect eigenvalue multiplicities. Recent studies such as \cite{Belardo2019,Ramezani2020,Khan2025,Ciampella2025} have developed spectral characterizations for signed graphs, including results on nullity, balance, and extremal configurations. Extending multiplicity bounds from unsigned graphs to signed graphs introduces additional challenges: cycle signs interact with eigenvectors, switching can alter the adjacency matrix without changing the spectrum, and balance imposes structural constraints not present in the unsigned setting.

\medskip
\noindent{\bf Motivation and gap in the literature.}
While the nullity of signed graphs (the case $\lambda=0$) has been thoroughly studied, far less is known about the multiplicity of an \emph{arbitrary} eigenvalue $\lambda$. A major recent advance is due to Wu et al.\ \cite{Wu2022}, who proved a sharp bound for $m(G^\sigma,0)$ in terms of girth and characterized all cases of equality. However, their result applies only to the eigenvalue $0$. Extending the same bound to all eigenvalues is nontrivial: general eigenvalues lack direct combinatorial interpretations, interact differently with signed cycles, and require a deeper analysis of switching classes and induced substructures.

\medskip
\noindent{\bf Contributions of this paper.}
In this work, we establish a sharp upper bound on the multiplicity $m(G^\sigma,\lambda)$ of \emph{any} real eigenvalue $\lambda$ of a connected signed graph $G^\sigma$ of order $n$ and girth $g(G^\sigma)$:
\[
m(G^\sigma,\lambda) \leq n - g(G^\sigma) + 2.
\]
We completely determine all signed graphs achieving equality. Equality holds if and only if $G^\sigma$ is:
\begin{itemize}
    \item a balanced complete signed graph $K_n^\sigma$ with $\lambda = -1$,
    \item an antibalanced complete signed graph $K_n^\sigma$ with $\lambda = 1$, or
    \item a balanced complete bipartite signed graph $K_{n_1,n_2}^\sigma$ (with $n_1,n_2\ge 2$) with $\lambda = 0$.
\end{itemize}
We also fully characterize signed cycles satisfying $m(G^\sigma,\lambda)=n-g(G^\sigma)+q$ for $q\in\{1,2\}$. Our results extend and generalize those of Wu et al.\ \cite{Wu2022}, whose theorem applies only to $\lambda=0$.

For convenience, we restate their result here:

\begin{theorem}\label{Theorem1} \cite{Wu2022}
Let $G^\sigma$ be a connected signed graph of order $n$ and girth $g(G^\sigma)$. Then
$m(G^\sigma,0) \leq n-g(G^\sigma)+2$, with equality if and only if $G^\sigma$ is one of the following:
(i) a balanced complete bipartite signed graph with $g(G^\sigma) = 4$;
(ii) a positive (resp., negative) signed cycle $(C_g,\sigma)$ with $g(G^\sigma)\equiv 0 \pmod{4}$ (resp., $g(G^\sigma)\equiv 2 \pmod{4}$).
\end{theorem}

\medskip
\noindent{\bf Structure of the paper.}
Section~2 provides essential definitions and preliminaries.
Section~3 establishes the main multiplicity bound and the equality characterization.
Section~4 fully characterizes signed cycles with $m(G^\sigma,\lambda)=n-g(G^\sigma)+q$ for $q=1,2$.

\section{Preliminaries}

In this section we recall fundamental concepts and results from spectral graph theory and the theory of signed graphs that will be used throughout the paper. We begin with basic notions of signed cycles, balance, and switching equivalence.

\medskip
\noindent{\bf Signed cycles and balance.}
Let $C^\sigma$ be a signed cycle. Its \emph{signature} (or sign) is defined as the product of the signs of its edges:
\[
\operatorname{sgn}(C^\sigma)=\prod_{e\in E(C^\sigma)}\sigma(e).
\]
A cycle is called \emph{balanced} if $\operatorname{sgn}(C^\sigma)=+1$ and \emph{unbalanced} otherwise.
A signed graph $G^\sigma$ is \emph{balanced} if every cycle in $G^\sigma$ is balanced; otherwise it is \emph{unbalanced}.
A signed graph is \emph{antibalanced} if every even cycle is positive and every odd cycle is negative.

\medskip
\noindent{\bf Switching equivalence.}
Given a signed graph $G^\sigma=(G,\sigma)$ and a \emph{switching function} $\zeta:V\to\{\pm1\}$, the signed graph $G^{\sigma^\zeta}=(G,\sigma^\zeta)$ is obtained by switching at the vertices according to
\[
\sigma^\zeta(uv)=\zeta(u)\,\sigma(uv)\,\zeta(v).
\]
Switching preserves the sign of every cycle, and two signed graphs $G^{\sigma_1}$ and $G^{\sigma_2}$ on the same underlying graph are \emph{switching equivalent}, written $G^{\sigma_1}\sim G^{\sigma_2}$, if one can be obtained from the other via switching at a set of vertices.
Switching equivalence plays a role analogous to isomorphism for simple graphs. An illustration of switching equivalent graphs is shown in Figure~\ref{Figeq}.

\begin{figure}[htbp]
    \centering
    \includegraphics[width=0.6\textwidth]{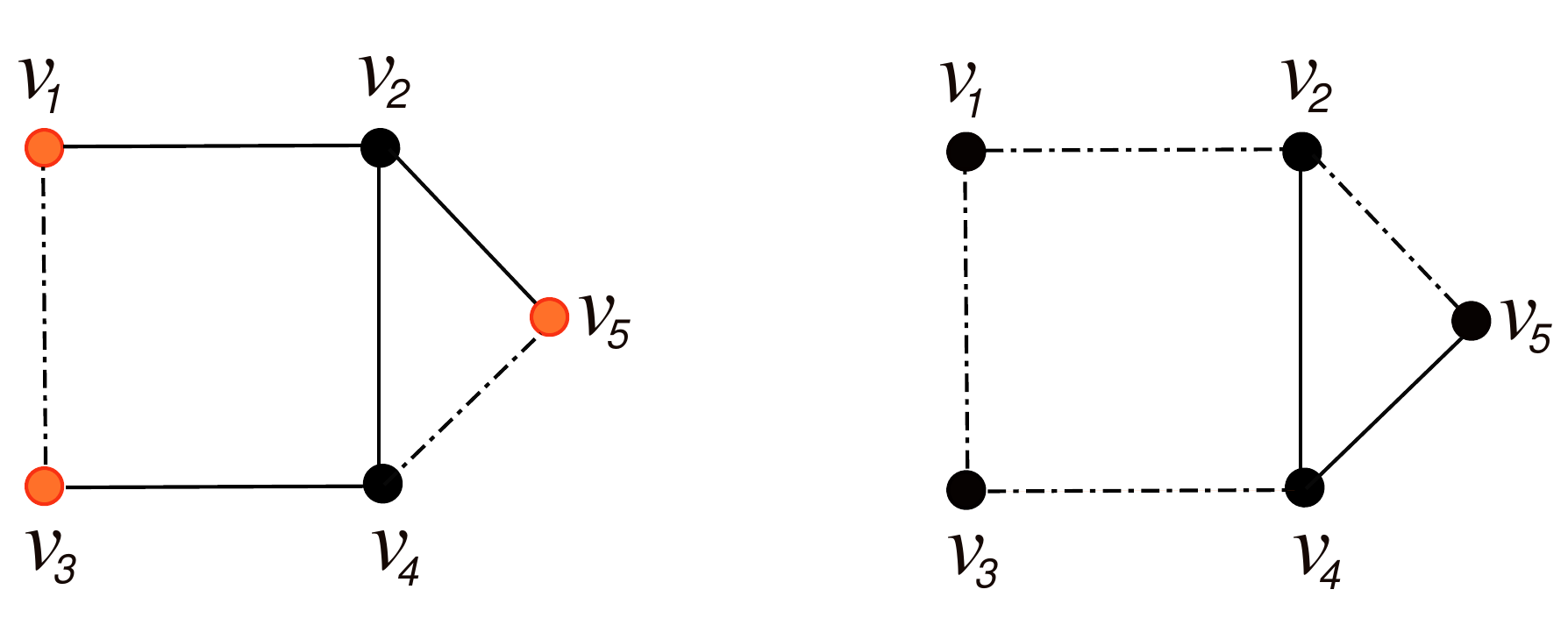}
    \caption{An example of switching equivalent graphs switched on $U = \{v_1, v_3, v_5\}$.}
    \label{Figeq}
\end{figure}

\medskip
\noindent{\bf Adjacency matrix of a signed graph.}
For a signed graph $G^\sigma$ of order $n$, the adjacency matrix $A(G^\sigma)=[a^{\sigma}_{ij}]$ is defined by
\[
a^{\sigma}_{ij} =
\begin{cases}
\sigma(v_i v_j), & \text{if } v_i\sim v_j,\\[4pt]
0, & \text{otherwise}.
\end{cases}
\]
Thus the sign pattern of $G^\sigma$ is encoded directly in the entries of $A(G^\sigma)$.
Since $A(G^\sigma)$ is real and symmetric, all eigenvalues are real and may be ordered as
\[
\lambda_1 \ge \lambda_2 \ge \cdots \ge \lambda_n.
\]
The adjacency spectrum of $G^\sigma$, denoted $\mathrm{Spec}_A(G,\sigma)$, is the multiset of eigenvalues of $A(G^\sigma)$ counted with algebraic multiplicity.
For an eigenvalue $\lambda$, its multiplicity is denoted $m(G^\sigma,\lambda)$.
Because $A(G^\sigma)$ is diagonalizable, the algebraic and geometric multiplicities coincide.
The rank of a signed graph is $\rk(G^\sigma)=\rk(A(G^\sigma))$, and for any eigenvalue $\lambda\in\mathbb{R}$,
\[
m(G^\sigma,\lambda)= n - \rk(\lambda I - A(G^\sigma)).
\]

\medskip
\noindent{\bf Balance via switching.}
A classical result of Harary, restated by Zaslavsky in \cite{Horn1985}, characterizes balance in terms of switching:

\begin{lemma}\label{Lemma1}
Let $G^\sigma=(G,\sigma)$ be a signed graph.
Then $G^\sigma$ is balanced if and only if $(G,\sigma)\sim(G,+)$, where $(G,+)$ is the all-positive signing of $G$.
\end{lemma}

\medskip
\noindent{\bf Interlacing.}
We recall the classical Cauchy interlacing theorem (see \cite{Horn1985}), which will be used repeatedly to bound eigenvalue multiplicities.

\begin{theorem}[Cauchy's Interlacing Theorem]
Let $C$ be a real symmetric $n\times n$ matrix and let $D$ be an $(n-1)\times(n-1)$ principal submatrix of $C$.
If $\lambda_1\ge\cdots\ge\lambda_n$ and $\mu_1\ge\cdots\ge\mu_{n-1}$ are the eigenvalues of $C$ and $D$ respectively, then
\[
\lambda_j \ge \mu_j \ge \lambda_{j+1}, \qquad j=1,\dots,n-1.
\]
\end{theorem}

Since the adjacency matrix of an induced subgraph is a principal submatrix of the adjacency matrix of the full graph, we immediately obtain the interlacing theorem for signed graphs:

\begin{lemma}[Interlacing Theorem for signed graphs]\label{interlacing}
Let $G^\sigma$ be a signed graph of order $n$ with eigenvalues $\lambda_1 \ge \dots \ge \lambda_n$, and let $H^\sigma$ be an induced subgraph of $G^\sigma$ with $m$ vertices and eigenvalues $\mu_1 \ge \dots \ge \mu_m$. Then
\[
\lambda_j \ge \mu_j \ge \lambda_{n-m+j}, \qquad j=1,2,\dots,m.
\]
\end{lemma}

This lemma implies that the eigenvalues of any induced subgraph ``interlace'' those of the original signed graph, and it will be a key tool for deriving upper bounds on eigenvalue multiplicities in the sequel.


\section{Bounds on eigenvalue multiplicities of signed graphs}

Throughout this section we assume that $g(G^\sigma) < n$, thereby excluding the case in which
$G^\sigma$ itself is a $g(G^\sigma)$-cycle. That case will be treated separately in Section~4.

\begin{lemma}\label{LemmaKn}
Let $G$ be a connected graph of order $n$ and girth $g(G)$. Then
\[
m(G,\lambda)=n-g(G)+2
\]
for some $\lambda\in\mathbb{R}$ if and only if $G\cong K_n$ $(n\ge 3)$ and $\lambda=-1$.
\end{lemma}

\begin{proof}
For $K_n$ we have $A(K_n)=J-I_n$, where $J$ is the all-ones matrix.
Since $\mathrm{Spec}(J)=\{0^{(n-1)}, n^{(1)}\}$ and $J$ commutes with $I_n$, it follows that
\[
\mathrm{Spec}(K_n)=\{(n-1)^{(1)},\, (-1)^{(n-1)}\}.
\]
Thus $m(K_n,-1)=n-1$.
Because $g(K_n)=3$, we obtain
\[
m(K_n,-1)=n-1=n-3+2=n-g(K_n)+2.
\]

Conversely, assume $m(G,\lambda)=n-g(G)+2$.
For $\lambda$ to have multiplicity $n-1$, the only connected graphs with an eigenvalue of multiplicity $n-1$ are $K_n$ and the star $K_{1,n-1}$.
Since $K_{1,n-1}$ has no cycles (girth undefined), it cannot satisfy $g(G)=3$.
Hence $G\cong K_n$, and the eigenvalue of multiplicity $n-1$ is $\lambda=-1$.
\end{proof}

\medskip

It is well known (see \cite{Sciriha1999,Cheng2007}) that the connected graphs of rank~$2$ are exactly complete bipartite graphs:

\begin{lemma}\cite{Sciriha1999,Cheng2007}\label{1}
Let $G$ be a connected graph of order $n$.
Then $\operatorname{rk}(G)=2$ if and only if $G\cong K_{n_1,n_2}$ with $n_1,n_2\ge 1$.
\end{lemma}

\medskip

\begin{lemma}\label{LemmaKn1n2}
Let $G$ be a connected graph of order $n$ and girth $g(G)$.
Then
\[
m(G,\lambda)=n-g(G)+2
\]
for some $\lambda\in\mathbb{R}$ if and only if $G\cong K_{n_1,n_2}$ with $n_1,n_2>1$ and $\lambda=0$.
\end{lemma}

\begin{proof}
\textit{Sufficiency.}
If $G\cong K_{n_1,n_2}$ with $n_1,n_2>1$, then $g(G)=4$ and
\[
\mathrm{Spec}(G)=\{\sqrt{n_1n_2},\, -\sqrt{n_1n_2},\, 0^{(n-2)}\}.
\]
Thus $m(G,0)=n-2=n-4+2=n-g(G)+2$.

\medskip
\textit{Necessity.}
Assume $m(G,\lambda)=n-g+2$. Then
\[
\rk(\lambda I-A(G))=g-2.
\]
Let $C$ be a shortest cycle of $G$.
Since $A(C)$ is a principal submatrix of $A(G)$, interlacing implies
\[
\rk(\lambda I-A(C))\le g-2.
\]
The only cycle for which $\rk(\lambda I-A(C))\le g-2$ is a $4$-cycle, because for $g\ge 5$ the matrix
$\lambda I-A(C)$ has rank at least $g-1$.
Hence $g=4$, so
\[
\rk(\lambda I-A(G))=2.
\]
Thus $\rk(A(G))=2$, and by Lemma~\ref{1}, $G\cong K_{n_1,n_2}$.
Since both parts must contain at least two vertices to have girth $4$, we obtain $n_1,n_2>1$.
Finally, because the only eigenvalue of multiplicity $n-2$ in $\mathrm{Spec}(K_{n_1,n_2})$ is $0$, we conclude $\lambda=0$.
\end{proof}

\medskip

\begin{theorem}\label{Theorem2}
Let $G^\sigma$ be a connected signed graph of order $n\ge 3$ and girth $g(G^\sigma)$.
Then for any real eigenvalue $\lambda$,
\[
m(G^\sigma,\lambda)\le n-g(G^\sigma)+2,
\]
and equality holds if and only if
\begin{enumerate}[label=(\roman*)]
    \item $G^\sigma$ is a balanced complete signed graph $K_n^\sigma$ with $\lambda=-1$;
    \item $G^\sigma$ is an antibalanced complete signed graph $K_n^\sigma$ with $\lambda=1$;
    \item $G^\sigma$ is a balanced complete bipartite signed graph $K_{n_1,n_2}^\sigma$ with $n_1,n_2\ge 2$ and $\lambda=0$.
\end{enumerate}
\end{theorem}

\begin{proof}
Let $\lambda$ be an eigenvalue of $A(G^\sigma)$.
By interlacing (Lemma~\ref{interlacing}) applied to any induced $g(G^\sigma)$-cycle,
\[
m(G^\sigma,\lambda)\le n-g(G^\sigma)+2.
\]

Assume equality holds.
We distinguish cases according to the switching class of $G^\sigma$.

\medskip
\noindent{\bf Case 1: $G^\sigma$ is balanced.}
By Lemma~\ref{Lemma1}, $G^\sigma$ is switching-equivalent to $(G,+)$, so
$m(G^\sigma,\lambda)=m(G,\lambda)$.
Applying Lemma~\ref{LemmaKn}, the only connected graph achieving the bound is $K_n$ with $\lambda=-1$.
Thus $G^\sigma$ is a balanced complete signed graph $K_n^\sigma$.

\medskip
\noindent{\bf Case 2: $G^\sigma$ is antibalanced.}
Switching equivalence yields $G^\sigma\sim (G,-)$, where every edge is negative.
Then
\[
A(G^\sigma) = -A(G),
\]
and hence $\lambda$ has multiplicity $n-1$ if and only if $-\lambda$ has multiplicity $n-1$ in $G$.
By Lemma~\ref{LemmaKn}, this occurs exactly when $G\cong K_n$ and $-\lambda=-1$, i.e.\ $\lambda=1$.
Thus $G^\sigma$ is an antibalanced complete signed graph.

\medskip
\noindent{\bf Case 3: $G^\sigma$ contains no triangles (so $g(G^\sigma)=4$).}
If $G^\sigma$ is balanced, we may switch to $(G,+)$.
Applying Lemma~\ref{LemmaKn1n2} to $G$, it follows that $G\cong K_{n_1,n_2}$ with $n_1,n_2\ge 2$ and $\lambda=0$.
Thus $G^\sigma$ is a balanced complete bipartite signed graph.

\medskip
These three cases exhaust all possibilities, completing the proof.
\end{proof}

\section{Characterization of Signed Cycles with $m(G^\sigma,\lambda)=n-g(G^\sigma)+q$ for $q=1,2$}

In this section we characterize the eigenvalues of signed cycles whose multiplicities satisfy
\[
m(G^\sigma,\lambda)=n-g(G^\sigma)+q,\qquad q\in\{1,2\}.
\]
Since a cycle $C_n^\sigma$ has $g(C_n^\sigma)=n$, the condition simplifies to
\[
m(C_n^\sigma,\lambda)=q.
\]
Thus we determine all eigenvalues of signed cycles with multiplicity $1$ and $2$.
The analysis follows directly from the explicit eigenvalue formulas for signed cycles.

\medskip
\noindent
We recall the well–known spectral formulas (see \cite{Belardo2015}):

\begin{lemma}\cite{Belardo2015}\label{Lemma3.7}
Let $C_n^\sigma$ and $P_n^\sigma$ denote a signed cycle and a signed path on $n$ vertices, respectively. Then:
\begin{enumerate}[label=(\roman*)]
\item $\displaystyle \mathrm{Spec}_A(C_n,+)=\bigl\{\,2\cos(2\pi j/n)\,\big|\,j=0,1,\dots,n-1\,\bigr\}$.
\item $\displaystyle \mathrm{Spec}_A(C_n,\bar\sigma)=\bigl\{\,2\cos((2j+1)\pi/n)\,\big|\,j=0,1,\dots,n-1\,\bigr\}$.
\item $\displaystyle \mathrm{Spec}_A(P_n,\sigma)=\bigl\{\,2\cos(j\pi/(n+1))\,\big|\,j=1,2,\dots,n\,\bigr\}$.
\end{enumerate}
\end{lemma}

These formulas allow a complete characterization of multiplicity-one and multiplicity-two eigenvalues of signed cycles.

\medskip

\begin{theorem}\label{Theorem3}
Let $G^\sigma$ be a connected signed graph of order $n$ and girth $g(G^\sigma)$.
Then
\[
m(G^\sigma,\lambda)=n-g(G^\sigma)+1
\]
for some real $\lambda$ if and only if $G^\sigma$ is one of the following:
\begin{enumerate}[label=(\roman*)]
\item a balanced cycle with $\lambda=2$;
\item a balanced cycle of even order with $\lambda=-2$;
\item an unbalanced cycle of odd order with $\lambda=-2$.
\end{enumerate}
\end{theorem}

\begin{proof}
Since $m(G^\sigma,\lambda)=n-g(G^\sigma)+1$, and a cycle satisfies $g(C_n^\sigma)=n$, we obtain
\[
m(C_n^\sigma,\lambda)=1.
\]
Thus $\lambda$ must be a simple (multiplicity-one) eigenvalue of $C_n^\sigma$.

\medskip
\noindent{\bf Balanced cycle.}
From Lemma~\ref{Lemma3.7}(i),
\[
\lambda_j = 2\cos(2\pi j/n).
\]
The symmetry $\lambda_j=\lambda_{n-j}$ implies that all eigenvalues occur in pairs except when
\[
j=n-j \pmod{n}.
\]
This occurs for:
\[
j=0\quad\Rightarrow\quad \lambda=2,
\]
and for $j=n/2$ when $n$ is even, giving
\[
\lambda=-2.
\]
Thus the balanced cycle has simple eigenvalues
\[
\lambda=2,\qquad \text{and }\lambda=-2\text{ for even }n.
\]

\medskip
\noindent{\bf Unbalanced cycle.}
From Lemma~\ref{Lemma3.7}(ii),
\[
\lambda_j = 2\cos\bigl((2j+1)\pi/n\bigr).
\]
Again, $\lambda_j=\lambda_{n-1-j}$, so multiplicity one occurs exactly when
\[
j=n-1-j\quad\Longleftrightarrow\quad 2j=n-1,
\]
which requires $n$ odd.
The corresponding eigenvalue is
\[
\lambda = 2\cos(\pi) = -2.
\]

\medskip
Combining the two cases yields the three families listed in the statement.
\end{proof}

\medskip

\begin{theorem}\label{Theorem4}
Let $G^\sigma$ be a connected signed graph of order $n$ and girth $g(G^\sigma)$.
Then
\[
m(G^\sigma,\lambda)=n-g(G^\sigma)+2
\]
for some real $\lambda$ if and only if $G^\sigma$ is one of the following:
\begin{enumerate}[label=(\roman*)]
\item a balanced cycle with
\[
\lambda=2\cos\Bigl(\frac{2\pi j}{g}\Bigr), \qquad j=1,2,\dots,\Bigl\lceil\frac{g}{2}\Bigr\rceil-1;
\]
\item an unbalanced cycle with
\[
\lambda=2\cos\Bigl(\frac{(2j+1)\pi}{g}\Bigr),\qquad
j=0,1,\dots,\Bigl\lceil\frac{g-1}{2}\Bigr\rceil-1.
\]
\end{enumerate}
\end{theorem}

\begin{proof}
Since $g(C_n^\sigma)=n$, the condition becomes
\[
m(C_n^\sigma,\lambda)=2.
\]
We use the explicit eigenvalues of signed cycles.

\medskip
\noindent{\bf Balanced cycle.}
Balanced cycle eigenvalues (Lemma~\ref{Lemma3.7}(i)) satisfy
\[
\lambda_j = 2\cos\Bigl(\frac{2\pi j}{n}\Bigr),
\qquad \lambda_j=\lambda_{n-j}.
\]
Thus every eigenvalue has multiplicity $2$ except:
\[
j=0\quad(\lambda=2),\qquad
j=n/2\text{ when $n$ is even }(\lambda=-2).
\]
Hence the multiplicity-two eigenvalues are
\[
\lambda_j=2\cos\Bigl(\frac{2\pi j}{n}\Bigr),
\qquad j=1,\dots,\Bigl\lceil\frac{n}{2}\Bigr\rceil-1.
\]

\medskip
\noindent{\bf Unbalanced cycle.}
Unbalanced cycle eigenvalues (Lemma~\ref{Lemma3.7}(ii)) satisfy
\[
\lambda_j = 2\cos\Bigl(\frac{(2j+1)\pi}{n}\Bigr),
\qquad \lambda_j=\lambda_{n-1-j}.
\]
Thus every eigenvalue has multiplicity $2$ except the possible simple eigenvalue $-2$ when $n$ is odd.
Hence the multiplicity-two eigenvalues are
\[
\lambda_j = 2\cos\Bigl(\frac{(2j+1)\pi}{n}\Bigr),
\qquad j=0,\dots,\Bigl\lceil\frac{n-1}{2}\Bigr\rceil-1.
\]

\medskip
Replacing $n$ by $g$ in both cases yields the statements (i) and (ii).
The converse is immediate from Lemma~\ref{Lemma3.7}.
\end{proof}

\section{Conclusion and Future Directions}

In this paper we established a sharp upper bound for the multiplicity of any real eigenvalue
$\lambda$ of a connected signed graph $G^\sigma$ in terms of its girth. Specifically, we showed that
\[
m(G^\sigma,\lambda) \le n - g(G^\sigma) + 2,
\]
and we completely characterized all signed graphs attaining equality.
The extremal cases are precisely the balanced and antibalanced complete signed graphs, as well as
the balanced complete bipartite signed graphs, corresponding to the eigenvalues $-1$, $1$, and $0$,
respectively. These results fully generalize the theorem of Wu et al.~\cite{Wu2022}, which treated
only the special case $\lambda=0$, and they reveal a unified structural principle: equality holds
exactly when the adjacency matrix exhibits minimal possible rank compatible with the girth.

We further analyzed signed cycles in detail and determined all eigenvalues that occur with
multiplicity $1$ or $2$. Using explicit spectral formulas, we obtained complete characterizations of
balanced and unbalanced cycles realizing
\[
m(G^\sigma,\lambda)=n-g(G^\sigma)+q,\qquad q\in\{1,2\}.
\]
These results illustrate how the interplay between cycle parity, balance, and switching equivalence
governs the spectrum of signed cycles.

\medskip
\noindent{\bf Future work.}
The results presented here suggest several natural directions for further research:

\begin{itemize}
    \item {\em Laplacian eigenvalue multiplicities.}
    Extending our girth-based multiplicity bounds to the signed Laplacian or the $A_\alpha$-matrix
    may reveal new relationships between structural balance and spectral extremality.

    \item {\em Beyond girth: cycle rank and cyclomatic number.}
    Since the cyclomatic number measures the complexity of cycle space, it would be interesting to
    derive multiplicity bounds involving both girth and cyclomatic number, generalizing results such
    as those of Ma et al.~\cite{Ma2016} and Tam–Huang~\cite{Tam2017}.

    \item {\em Higher multiplicity phenomena.}
    The extremal patterns identified for complete and complete bipartite signed graphs raise the
    question of whether similar structural characterizations exist for multiplicities of the form
    $m(G^\sigma,\lambda)=n-g(G^\sigma)+k$ for $k>2$.

    \item {\em Random and sparse signed graphs.}
    Investigating typical multiplicity behavior in random signed graphs or sparse models may provide
    probabilistic analogues of our extremal results.

    \item {\em Algorithmic detection of extremal switching classes.}
    Since the extremal graphs are defined up to switching, developing efficient algorithms to detect
    whether a signed graph lies in one of the extremal switching classes would have practical
    applications in network analysis.
\end{itemize}

Overall, the results of this paper highlight how cycle structure, balance, and switching equivalence
interact to control eigenvalue multiplicities in signed graphs. We hope that the techniques
developed here will serve as a foundation for further investigations in signed spectral graph theory.


\end{document}